\documentclass[11pt]{amsart}

\usepackage{amsmath, amsthm, amssymb, amsfonts, amscd}

\def\AA{{\mathbb A}}

\def\CC{{\mathbb C}}
\def\FF{{\mathbb F}}
\def\GG{{\mathbb G}}
\def\KK{{\mathbb K}}

\def\NN{{\mathbb N}}

\def\QQ{{\mathbb Q}}
\def\PP{{\mathbb P}}
\def\QQ{{\mathbb Q}}
\def\RR{{\mathbb R}}
\def\TT{{\mathbb T}}
\def\ZZ{{\mathbb Z}}

\def\fp{{\mathfrak p}}

\def\0{{\mathbf 0}}
\def\1{{\mathbf 1}}

\def\Abf{{\mathbf A}}

\def\Pbf{{\mathbf P}}

\def\Sbf{{\mathbf S}}
\def\Tbf{{\mathbf T}}

\def\Acal{{\mathcal A}}

\def\Ccal{{\mathcal C}}

\def\Ocal{{\mathcal O}}

\def\Xcal{{\mathcal X}}

\def\ch{\mathrm{char}}

\def\Gal{\mathrm{Gal}}

\def\Spec{\mathrm{Spec}}

\def\Proj{\mathrm{Proj}}

\def\sch{\mathrm{sch}}

\def\sup{\mathrm{sup}}

\theoremstyle{plain}

\newtheorem{thm}{Theorem}
\newtheorem{cor}[thm]{Corollary}
\newtheorem{prop}[thm]{Proposition}
\newtheorem{lem}[thm]{Lemma}

\theoremstyle{remark}

\begin{document}

\title[Weak Convergence on Berkovich Projective Space]{A Criterion for Weak Convergence on Berkovich Projective Space}
\author{Clayton Petsche}

\address{Clayton Petsche; Department of Mathematics and Statistics; Hunter College; 695 Park Avenue; New York, NY 10065 U.S.A.}
\email{cpetsche@hunter.cuny.edu}

\thanks{The author is partially supported by NSF Grant DMS-0901147.}
\begin{abstract}
We give a criterion for the weak convergence of unit Borel measures on the $N$-dimensional Berkovich projective space $\Pbf^{N}_K$ over a complete non-archimedean field $K$.  As an application, we give a sufficient condition for a certain type of equidistribution on $\Pbf^{N}_K$ in terms of a weak Zariski-density property on the scheme-theoretic projective space $\PP^N_{\tilde{K}}$ over the residue field $\tilde{K}$.  As a second application, in the case of residue characteristic zero we give an ergodic-theoretic equidistribution result for  the powers of a point $a$ in the $N$-dimensional unit torus $\TT^N_K$ over $K$.  This is a non-archimedean analogue of a well-known result of Weyl over $\CC$, and its proof makes essential use of a theorem of Mordell-Lang type for $\GG_m^N$ due to Laurent.
\end{abstract}

\maketitle


\section{Introduction}  

\subsection{}  Let $K$ be a field which is complete with respect to a nontrivial, non-archimedean absolute value.  Given an integer $N\geq1$, the $N$-dimensional projective space $\PP^N(K)$ is compact (with respect to its Hausdorff analytic topology) if and only if the field $K$ is locally compact, and this occurs only when $K$ has both a discrete value group and a finite residue field.  On the other hand, the $N$-dimensional Berkovich projective space $\Pbf^{N}_K$ over $K$ is a Hausdorff space which contains the ordinary projective space $\PP^N(K)$ as a subspace, and $\Pbf^{N}_K$ is always compact, regardless of whether or not $K$ is locally compact.  Moreover, the Hausdorff topology on $\Pbf^{N}_K$ is closely related not only to the analytic topology on $\PP^N(K)$, but also to the Zariski topology on the scheme-theoretic projective space $\PP^N_{\tilde{K}}$ over the residue field $\tilde{K}$.  For these and other reasons, there are many situations in which it is preferable to work on the larger space $\Pbf^{N}_K$ rather than $\PP^N(K)$ itself.

An example of an analytic notion which is best studied on compact spaces is that of equidistribution.  For each integer $\ell\geq1$, let $Z_\ell$ be a finite multiset of points in $\Pbf^{N}_K$ (a multiset is a set whose points occur with multiplicities), and let $\mu$ be a unit Borel measure on $\Pbf^{N}_K$.  The sequence $\langle Z_\ell\rangle_{\ell=1}^{+\infty}$ is said to be $\mu$-equidistributed if the limit
\begin{equation}\label{IntroEqui}
\lim_{\ell\to+\infty}\frac{1}{|Z_\ell|}\sum_{z\in Z_\ell}\varphi(z) = \int\varphi d\mu
\end{equation}
holds for all continuous functions $\varphi:\Pbf^{N}_K\to\RR$.  Many researchers have been working recently to establish equidistribution results on Berkovich analytic spaces, often for sequences $\langle Z_\ell\rangle_{\ell=1}^{+\infty}$ arising naturally from questions in arithmetic geometry and dynamical systems; we mention \cite{BakerPetsche}, \cite{BakerRumely}, \cite{ChambertLoir}, \cite{Faber}, \cite{FavreRiveraLetelier}, and \cite{Gubler} to name a few examples of such work.

The main result of this paper implies that, in order to establish the equidistribution of a sequence $\langle Z_\ell\rangle_{\ell=1}^{+\infty}$ of multisets in $\Pbf^{N}_K$, it suffices to check $(\ref{IntroEqui})$ for a special class of continuous functions $\varphi$ arising naturally from the geometric structure of $\Pbf^{N}_K$.  Since it requires no extra work to do so, we formulate our main result using the more general notion of weak convergence of measures on $\Pbf^{N}_K$, although our motivating concern is with equidistribution.  Moreover, for reasons which we will discuss below, our main result will be stated using nets, rather than sequences; again this requires no extra effort.

\medskip

We give two applications of our main result, both of which establish equidistribution theorems with respect to the Dirac measure $\delta_\gamma$ supported at the Gauss point $\gamma$ of $\Pbf^{N}_K$.  Letting $\PP^N_{\tilde{K}}$ denote the scheme-theoretic projective space over the residue field $\tilde{K}$, there exists a natural reduction map $r:\Pbf^{N}_K\to\PP^N_{\tilde{K}}$, and the Gauss point $\gamma$ can be described as the unique point of $\Pbf^{N}_K$ reducing to the generic point of $\PP^N_{\tilde{K}}$.  Our first application gives a useful necessary and sufficient condition for $\delta_\gamma$-equidistribution, and uses this to establish the $\delta_\gamma$-equidistribution of nets whose reduction satisfy a certain weak Zariski-density property.

\medskip

Our second application is an ergodic-theoretic equidistribution result for the sequence formed by taking the powers of a point in the $N$-dimensional unit torus 
\begin{equation*}
\TT^N_K=\{(a_0:a_1:\dots:a_N)\in\PP^N(K) \mid |a_0|=|a_1|=\dots=|a_N|=1\}
\end{equation*}
over $K$.  Identifying the group variety $\GG_m^N$ over the residue field $\tilde{K}$ with the subvariety of $\PP^N$ defined by $x_0x_1\dots x_N\neq0$, the reduction map $r:\Pbf^{N}_K\to\PP^N_{\tilde{K}}$ restricts to a map $r:\TT^N_K\to\GG_m^N(\tilde{K})$.  A point $\tilde{a}$ in $\GG_m^N(\tilde{K})$ is said to be non-degenerate if it is not contained in any proper algebraic subgroup of $\GG_m^N$.

\begin{thm}\label{NonArchErgIntro}
Assume that the residue field $\tilde{K}$ has characteristic zero.  Let $a\in\TT^N_K$, and for each integer $\ell\geq1$, define $Z_\ell=\{a,a^2,\dots,a^\ell\}$, considered as a multiset in $\TT^N_K\subset\PP^N(K)\subset\Pbf_K^N$ of cardinality $|Z_\ell|=\ell$.  The sequence $\langle Z_\ell\rangle_{\ell=1}^{\infty}$ is $\delta_\gamma$-equidistributed in $\Pbf_K^N$ if and only if the point $\tilde{a}$ is non-degenerate in $\GG_m^N(\tilde{K})$.
\end{thm}

This theorem is a non-archimedean analogue of a well-known archimedean equidistribution result of Weyl \cite{Weyl}.  Given a point $a$ in the compact unit torus $\TT^N_\CC$ over $\CC$, Weyl's result gives necessary and sufficient conditions for the Haar-equidistribution of the sets $Z_\ell=\{a,a^2,\dots,a^\ell\}$.  Usually stated in its additive (rather than multiplicative) form, this theorem is often presented as the first nontrivial example of  ``uniform distribution modulo 1''.  We will give a more detailed discussion of Weyl's theorem and related results in $\S$~\ref{ErgodicSect}.  An essential ingredient in our proof of Theorem~\ref{NonArchErgIntro} is a theorem of Mordell-Lang type on the group variety $\GG_m^N$, due to Laurent \cite{Laurent}.

\medskip

This paper is organized as follows:
\begin{itemize}
	\item In $\S$~\ref{PrelimSect} we fix some notation and terminology.
	\item In $\S$~\ref{AffineSect} we review the definitions of the Berkovich affine and projective spaces $\Abf^{N+1}_K$ and $\Pbf^{N}_K$, and we establish the needed topological properties of these spaces.  
	\item In $\S$~\ref{MainThmSect} we prove the main result of this paper, namely the criterion for weak convergence of unit Borel measures on $\Pbf^{N}_K$. 
	\item In $\S$~\ref{ReductionSect} and $\S$~\ref{ErgodicSect} we present the two applications.
\end{itemize}

The author would like to acknowledge the anonymous referee for his or her generous help, including the suggestion of Theorem~\ref{DeltaGammaEquiTheorem}, and X. Faber for several helpful suggestions.

\subsection{}\label{PrelimSect}  Throughout this paper $K$ denotes a field which is complete with respect to a nontrivial, non-archimedean absolute value $|\cdot|$.  Denote by $K^\circ=\{a\in K\mid |a|\leq1\}$ the valuation ring of $K$, by $K^{\circ\circ}=\{a\in K\mid |a|<1\}$ its maximal ideal, and by $\tilde{K}=K^\circ/K^{\circ\circ}$ its residue field.  Given an element $a\in K^\circ$, we denote by $\tilde{a}$ the image of $a$ under the quotient map $K^\circ\to\tilde{K}$.  Let $\KK$ be the completion of an algebraic closure of $K$ with respect to the unique extension of $|\cdot|$; thus $\KK$ is both complete and algebraically closed (\cite{BGR} $\S$3.4).  Define $\KK^\circ$, $\KK^{\circ\circ}$, and $\tilde{\KK}$ analogously.

Let $N\geq0$ be an integer, and let $K[X]=K[X_0,X_1,\dots,X_N]$ be the polynomial ring over $K$ in the $N+1$ variables $X=(X_0,X_1,\dots,X_N)$.  By a multiplicative seminorm on $K[X]$ extending $|\cdot|$ we mean a nonnegative real-valued function $[\cdot]$ on $K[X]$ satisfying $[a]=|a|$ for all constants $a\in K$, and satisfying $[f+g]\leq\max\{[f],[g]\}$ and $[fg]=[f][g]$ for all pairs $f,g\in K[X]$.

Given an arbitrary polynomial $f\in K[X]$, denote by $H(f)$ the maximum absolute value of the coefficients of $f$.  Thus $H(f)\leq1$ if and only if $f$ is defined over the valuation ring $K^\circ$; in this case we denote by $\tilde{f}\in\tilde{K}[X]$ the reduction of $f$.  If $H(f)=1$, we say that $f$ is normalized.

If the non-archimedean field $K$ has a countable dense subset, then it is possible to show using the Urysohn metrization theorem that the space $\Pbf^{N}_K$ is metrizable; see \cite{BakerRumelyBook} $\S$~1.5.  In general, however, $\Pbf^{N}_K$ is not homeomorphic to a metric space.  Consequently, notions of convergence in $\Pbf^{N}_K$ are best studied using nets, rather than sequences.

Briefly, a net in a set $\Xcal$ is a function $\alpha\mapsto x_\alpha$ from a directed set $I$ into $\Xcal$; it is usually denoted by $\langle x_\alpha\rangle$, suppressing the dependence on $I$.  A sequence in $\Xcal$ is simply a net in $\Xcal$ indexed by the directed set $\NN=\{1,2,3,\dots\}$ of positive integers.  In order to distinguish them from arbitrary nets, we will generally refer to sequences using the notation $\langle x_\ell\rangle_{\ell=1}^{+\infty}$.  If the set $\Xcal$ is a metric space, then many familiar topological concepts can be reformulated in terms of convergence properties of sequences in $\Xcal$.  These results continue to hold when $\Xcal$ is an arbitrary Hausdorff space, but only if one takes care to properly interpret the statements using nets in place of sequences.  We refer the reader to Folland \cite{Folland} $\S$~4.3 for a treatment of nets and their convergence properties.

\section{Berkovich Affine and Projective Space}\label{AffineSect}

\subsection{}\label{AffineSubSect}  The Berkovich affine space $\Abf^{N+1}_K$ over $K$ is defined to be the set of multiplicative seminorms on $K[X]$ extending $|\cdot|$.  As a matter of notation, we will refer to a point $\zeta\in\Abf^{N+1}_K$, and denote by $[\cdot]_\zeta$ its corresponding seminorm.  The topology on $\Abf^{N+1}_K$ is defined to be the weakest topology with respect to which the real-valued functions $\zeta\mapsto[f]_\zeta$ are continuous for all $f\in K[X]$.  Equivalently, define a family of subsets of $\Abf^{N+1}_K$ by 
\begin{equation*}
U_{s,t}(f) = \{\zeta\in\Abf^{N+1}_K \mid s<[f]_\zeta<t\} 
\end{equation*}
for $f\in K[X]$ and $s,t\in\RR$.  By definition, the subsets $U_{s,t}(f)$ generate a base of open sets for the topology on $\Abf^{N+1}_K$.

To see the relation between $\Abf^{N+1}_K$ and the classical affine space $K^{N+1}$, consider a point $a\in K^{N+1}$.  Letting $[\cdot]_a$ denote the multiplicative seminorm defined by evaluation $[f]_a=|f(a)|$, we obtain a continuous embedding $K^{N+1}\hookrightarrow\Abf^{N+1}_K$ given by $a\mapsto[\cdot]_a$.  We regard this embedding as an inclusion $K^{N+1}\subset\Abf^{N+1}_K$ by identifying $K^{N+1}$ with its image in $\Abf^{N+1}_K$.  More generally, evaluation $a\mapsto[\cdot]_a$ defines a map $\KK^{N+1}\to\Abf^{N+1}_K$ whose image is homeomorphic to the quotient of $\KK^{N+1}$ by the action of $\Gal(\KK/K)$.

The classical affine space $K^{N+1}$ is always a proper subset of $\Abf^{N+1}_K$.  For example, an important class of points in $\Abf^{N+1}_K$ arises by considering polydiscs
\begin{equation*}
D(c,r) = \{a\in \KK^{N+1} \mid |c_n-a_n|\leq r_n\text{ for all }0\leq n\leq N\}
\end{equation*}
with center $c\in \KK^{N+1}$ and polyradius $r\in|\KK|^{N+1}$.  The function $[\cdot]_{\zeta_{c,r}}:K[X]\to\RR$ defined by the supremum $[f]_{\zeta_{c,r}}:=\sup\{|f(a)|\mid a\in D(c,r) \}$ is a multiplicative seminorm on $K[X]$ extending $|\cdot|$, and therefore it defines a point of $\Abf^{N+1}_K$; it is convenient to denote this point by $\zeta_{c,r}$.

\subsection{}\label{AffineNormSect}  Define a function $\|\cdot\|:\Abf^{N+1}_K\to\RR$ by 
\begin{equation*}
\|\zeta\|= \max\{[X_0]_\zeta,[X_1]_\zeta,\dots,[X_N]_\zeta\}.
\end{equation*}
Observe that $\|\cdot\|$ is continuous, $\|\zeta\|\geq0$ for all $\zeta\in\Abf^{N+1}_K$, and $\|\zeta\|=0$ if and only if $\zeta$ is the point $0=(0,0,\dots,0)$ corresponding to the origin in $K^{N+1}\subset\Abf^{N+1}_K$.  For each real number $r>0$, define two subsets of $\Abf^{N+1}_K$ by
\begin{equation*}
E_r = \{\zeta\in\Abf^{N+1}_K\mid \|\zeta\|\leq r\} \hskip1cm U_r = \{\zeta\in\Abf^{N+1}_K\mid \|\zeta\|< r\}.
\end{equation*}
$U_r$ is clearly open, and in proving the following proposition we will see that $E_r$ is compact.  

\begin{prop}\label{AffineLocComp}
$\Abf^{N+1}_K$ is a locally compact Hausdorff space.
\end{prop}
\begin{proof}
To show that $\Abf^{N+1}_K$ is Hausdorff, let $\zeta,\zeta'\in\Abf^{N+1}_K$ be distinct points.  Then $[f]_\zeta\neq[f]_{\zeta'}$ for some $f\in K[X]$.  Selecting disjoint open intervals $(s,t)$ and $(s',t')$ containing $[f]_\zeta$ and $[f]_{\zeta'}$, respectively, it follows that $U_{s,t}(f)$ and $U_{s',t'}(f)$ are disjoint open neighborhoods of $\zeta$ and $\zeta'$, respectively.  

We will now show that the sets $E_r$ are compact; since $\Abf^{N+1}_K=\cup_{r>0}U_r$, it will follow at once that $\Abf^{N+1}_K$ is locally compact.  In order to show that $E_r$ is compact it suffices to show that every net $\langle \zeta_\alpha\rangle$ in $E_r$ has a subnet converging to a limit in $E_r$.  For each $f\in K[X]$, there exists a constant $C_{f,r}>0$ such that $[f]_\zeta\leq C_{f,r}$ for all $\zeta\in E_r$.  (Each seminorm $[\cdot]_\zeta$ satisfies the ultrametric inequality, so one could take $C_{f,r} = H(f)r^{\deg(f)}$.)  Therefore the association $\zeta\mapsto[f]_\zeta$ defines a continuous map $E_r\to[0,C_{f,r}]$.  Letting $\Pi$ denote the product $\prod_{f\in K[X]}[0,C_{f,r}]$, we obtain a continuous map $\iota:E_r\to\Pi$, which is injective by the definition of $\Abf^{N+1}_K$.  Since $\Pi$ is compact (Tychonoff's theorem, \cite{Folland} $\S$~4.6), the net $\langle \iota(\zeta_\alpha)\rangle$ has a subnet $\langle \iota(\zeta_\beta)\rangle$ converging to some point $(\xi_f)_{f\in K[X]}\in\Pi$.  Define a function $[\cdot]_\xi:K[X]\to\RR$ by $[f]_\xi=\xi_f$.  Then $[\cdot]_\xi$ is a multiplicative seminorm on $K[X]$ restricting to $|\cdot|$ on $K$, and thus it corresponds to a point $\xi\in\Abf^{N+1}_K$.  Moreover, $\xi\in E_r$ and $\zeta_\beta\to \xi$, as desired, completing the proof that $E_r$ is compact.
\end{proof}

\subsection{}\label{BerkProjSect}  Assume now that $N\geq1$.  Define an equivalence relation $\sim$ on $\Abf^{N+1}_K\setminus\{0\}$ by declaring that $\zeta\sim \xi$ if and only if there exists a constant $\lambda>0$ such that $[f]_\zeta=\lambda^{\deg(f)}[f]_\xi$ for all homogeneous polynomials $f\in K[X]$.  The Berkovich projective space $\Pbf^{N}_K$ is defined to be the quotient of $\Abf^{N+1}_K\setminus\{0\}$ modulo $\sim$, endowed with the quotient topology; denote by $\pi:\Abf^{N+1}_K\setminus\{0\}\to\Pbf^{N}_K$ the quotient map.

The embedding $K^{N+1}\hookrightarrow\Abf^{N+1}_K$ discussed in $\S$~\ref{AffineSect} restricts to a map $K^{N+1}\setminus\{0\}\to\Abf^{N+1}_K\setminus\{0\}$, which descends modulo $\sim$ to an embedding $\PP^N(K)\hookrightarrow\Pbf^N_K$.  We again regard this map as an inclusion $\PP^N(K)\subset\Pbf^N_K$ by identifying $\PP^N(K)$ with its image in $\Pbf^N_K$.  Similarly, the map $\KK^{N+1}\to\Abf^{N+1}_K$ descends modulo $\sim$ to a map $\PP^N(\KK)\hookrightarrow\Pbf^N_K$ whose image is homeomorphic to the quotient of $\PP^N(\KK)$ by the action of $\Gal(\KK/K)$.

Consider the subset of $\Abf^{N+1}_K$ defined by 
\begin{equation*}
\Sbf_K^N=\{\zeta\in\Abf^{N+1}_K\mid \|\zeta\|= 1\}.
\end{equation*}
Note that $\Sbf_K^N$ is compact, since $\Sbf_K^N=E_1\setminus U_1$ for the compact set $E_1$ and the open set $U_1$ defined in $\S$~\ref{AffineNormSect}.  The following lemma shows that the quotient map $\pi:\Abf^{N+1}_K\setminus\{0\}\to\Pbf^{N}_K$ remains surjective when restricted to $\Sbf_K^N$; we will still use the notation $\pi:\Sbf_K^N\to\Pbf^{N}_K$ to refer to this restricted map.

\begin{lem}\label{RestSurjLem}
Given a point $z\in\Pbf^{N}_K$, there exists a point $\zeta_z\in \Sbf_K^N$ such that $\pi(\zeta_z)=z$.
\end{lem}
\begin{proof}
Let $\zeta$ be an arbitrary point in $\Abf^{N+1}_K\setminus\{0\}$ such that $\pi(\zeta)=z$.  By symmetry, we may assume without loss of generality that $X_N$ is the coordinate at which the maximum $\|\zeta\|=[X_N]_\zeta$ is attained; thus $[X_N]_\zeta\neq0$.  Given $f\in K[X]$, observe that $X_N^\ell f(X_0/X_N, X_1/X_N, \dots,X_{N-1}/X_N,1)$ is a polynomial for all sufficiently large integers $\ell\geq0$.  Given such an integer $\ell$, define
\begin{equation*}
[f]_{\zeta_z} := [X_N]^{-\ell}_\zeta[X_N^\ell f(X_0/X_N, X_1/X_N, \dots,X_{N-1}/X_N,1)]_\zeta.
\end{equation*}
The value of $[f]_{\zeta_z}$ does not depend on $\ell$ since $[\cdot]_\zeta$ is multiplicative.  Moreover, $[\cdot]_{\zeta_z}$ inherits the axioms of a multiplicative seminorm from $[\cdot]_\zeta$, and if $f$ is homogeneous then $[f]_\zeta=[X_N]_\zeta^{\deg(f)}[f]_{\zeta_z}$; therefore $[\cdot]_{\zeta_z}$ defines an element $\zeta_z\in\Abf^{N+1}_K\setminus\{0\}$ with $\pi(\zeta_z)=\pi(\zeta)=z$.  Finally, we have $[X_N]_{\zeta_z}=[1]_\zeta=1$ and $[X_n]_{\zeta_z}=[X_N]_\zeta^{-1}[X_{n}]_\zeta\leq1$ for all $0\leq n\leq N-1$, whereby $\|\zeta_z\|=1$, and thus $\zeta_z\in  \Sbf_K^N$ as desired.
\end{proof}

\begin{prop}\label{ProjCompactProp}
$\Pbf^{N}_K$ is a compact Hausdorff space.
\end{prop}
\begin{proof}
The following is a standard result of general topology (\cite{Bourbaki} $\S$~10.2): if $S$ is a compact Hausdorff space and $f:S\to S'$ is a surjective quotient map onto a topological space $S'$, then $S'$ is Hausdorff if and only if the set $\{(x,y) \mid f(x)=f(y)\}$ is closed in $S\times S$.  

Applying this result to the map $\pi: \Sbf_K^N\to\Pbf^{N}_K$, in order to show that $\Pbf^{N}_K$ is Hausdorff it suffices to show that the set $R=\{(\zeta,\xi) \in  \Sbf_K^N\times  \Sbf_K^N\mid \zeta\sim \xi\}$ is closed in $ \Sbf_K^N\times  \Sbf_K^N$.  To show that $R$ is closed, consider a convergent net $\langle(\zeta_\alpha,\xi_\alpha)\rangle$ in $ \Sbf_K^N\times  \Sbf_K^N$ with $\zeta_\alpha\sim \xi_\alpha$ for all $\alpha\in A$, and with $(\zeta_\alpha,\xi_\alpha)\to(\zeta,\xi)\in  \Sbf_K^N\times  \Sbf_K^N$; we must show that $\zeta\sim \xi$.  By the definition of $\sim$, there exists a net of positive real numbers $\langle\lambda_\alpha\rangle$ such that $[f]_{\zeta_\alpha}=\lambda_\alpha^{\deg(f)}[f]_{\xi_\alpha}$ for all homogeneous $f\in K[X]$ and all $\alpha\in A$.  Since $\zeta_\alpha\to \zeta$ and $\xi_\alpha \to \xi$, and since the maps $\zeta\mapsto[f]_\zeta$ are continuous, we deduce that $[f]_{\zeta_\alpha}\to [f]_\zeta$ and $[f]_{\xi_\alpha}\to [f]_\xi$ for all $f\in K[X]$.  It follows that the net $\langle\lambda_\alpha\rangle$ converges to the number $\lambda:=([f]_\zeta/[f]_\xi)^{1/\deg(f)}$ for all homogeneous $f\in K[X]$.  Since $\RR$ is Hausdorff, the limit $\lambda$ is unique and therefore independent of $f$.  Since $[f]_{\zeta}=\lambda^{\deg(f)}[f]_{\xi}$, we deduce that $\zeta\sim \xi$.  This concludes the proof that $R$ is closed in $ \Sbf_K^N\times  \Sbf_K^N$, and therefore that $\Pbf^{N}_K$ is Hausdorff.  Since $ \Sbf_K^N$ is compact and $\pi: \Sbf_K^N\to\Pbf^{N}_K$ is continuous and surjective, $\Pbf^{N}_K$ must also be compact.
\end{proof}

\subsection{}  For background on Berkovich analytic spaces, see Berkovich's original monograph \cite{Berkovich}, especially $\S$~1.5 for a discussion of affine space $\Abf_K^{N+1}$.  The construction of $\Pbf_K^N$ via the equivalence relation $\sim$ on $\Abf_K^{N+1}$, which is similar to the scheme-theoretic $\Proj$ construction, is due to Berkovich himself \cite{Berkovich1}.  Baker-Rumely (\cite{BakerRumelyBook} $\S$~2.2) have treated the case $N=1$ at length, but for general $N\geq1$ the construction and basic topological properties of $\Pbf_K^N$ do not seem to have been written out in detail before now.  

The fundamental compactness argument, used here in the proof of Proposition~\ref{AffineLocComp}, is due to Baker-Rumely (\cite{BakerRumelyBook} Thm. C.3); it is slightly different than Berkovich's original argument (\cite{Berkovich} Thm. 1.2.1).  Naturally, both proofs use Tychonoff's theorem.

\section{A Criterion for Weak Convergence}\label{MainThmSect}

\subsection{}  Let $\Ccal(\Pbf^N_K)$ denote the space of continuous functions $\Pbf^N_K\to\RR$.  Thus $\Ccal(\Pbf^N_K)$ is a Banach algebra (with respect to the supremum norm).  By a Borel measure $\mu$ on $\Pbf^N_K$ we mean a positive measure on the Borel $\sigma$-algebra of $\Pbf^N_K$; we say $\mu$ is a unit Borel measure if $\mu(\Pbf^N_K)=1$.  Given a net $\langle\mu_\alpha\rangle$ of Borel measures on $\Pbf^N_K$, and and another Borel measure $\mu$ on $\Pbf^N_K$, we say that $\mu_\alpha\to\mu$ weakly if $\int \varphi d\mu_\alpha\to\int \varphi d\mu$ for all $\varphi\in\Ccal(\Pbf^N_K)$.  

We will now state and prove the main result of this paper.  Given a homogeneous polynomial $f\in K[X]$, it follows from the definition of the equivalence relation $\sim$ that the real-valued function $\zeta\mapsto[f]_\zeta/\|\zeta\|^{\deg(f)}$ on $\Abf_K^{N+1}$ is constant on $\sim$-equivalence classes.  We may therefore define the function
\begin{equation*}
\lambda_f:\Pbf^N_K\to\RR \hskip1cm \lambda_f(\pi(\zeta))=\frac{[f]_\zeta}{\|\zeta\|^{\deg(f)}}.
\end{equation*}

\begin{thm}\label{MainTheorem}
Let $\langle\mu_\alpha\rangle$ be a net of unit Borel measures on $\Pbf^N_K$, and let $\mu$ be another unit Borel measure on $\Pbf^N_K$.  Then $\mu_\alpha\to\mu$ weakly if and only if $\int\lambda_f d\mu_\alpha\to\int\lambda_f d\mu$ for all normalized homogeneous polynomials $f\in K[X]$.
\end{thm}

\begin{proof}
The ``only if'' direction is trivial since each function $\lambda_f$ is continuous.

To prove the ``if'' direction, assume that $\int\lambda_f d\mu_\alpha\to\int\lambda_f d\mu$ for all normalized homogeneous polynomials $f\in K[X]$.  Then in fact this limit must hold for arbitrary nonzero homogeneous $f\in K[X]$, which is easy to see by scaling $f$ and using the fact that $\lambda_{cf}=|c|\lambda_f$ for all nonzero $c\in K$.

Denote by $\Acal(\Pbf^N_K)$ the subspace of $\Ccal(\Pbf^N_K)$ generated over $\RR$ by the functions of the form $\lambda_f:\Pbf^N_K\to\RR$ for homogeneous $f\in K[X]$.  Then $\Acal(\Pbf^N_K)$ is a dense subalgebra of $\Ccal(\Pbf^N_K)$.  To see this, note that $\Acal(\Pbf^N_K)$ is closed under multiplication, since $\lambda_f\lambda_g=\lambda_{fg}$, and it is therefore a subalgebra. In order to show that $\Acal(\Pbf^N_K)$ is dense in $\Ccal(\Pbf^N_K)$, it suffices by the Stone-Weierstrass theorem (\cite{Folland} $\S$~4.7) to show that $\Acal(\Pbf^N_K)$ separates the points of $\Pbf^N_K$.  Consider two points $z,w\in\Pbf^{N}_K$ such that $\lambda_f(z)=\lambda_f(w)$ for all homogeneous $f\in K[X]$.  Taking $\zeta\in\pi^{-1}(z)$ and $\xi\in\pi^{-1}(w)$, we have $[f]_\zeta=(\|\zeta\|/\|\xi\|)^{\deg(f)}[f]_\xi$ for all homogeneous $f\in K[X]$, which means that $\zeta\sim \xi$, whereby $z=w$.  In other words, $z\neq w$ implies that  $\lambda_f(z)\neq\lambda_f(w)$ for some homogeneous $f\in K[X]$, showing that $\Acal(\Pbf^N_K)$ separates the points of $\Pbf^N_K$, and completing the proof that $\Acal(\Pbf^N_K)$ is dense in $\Ccal(\Pbf^N_K)$.

To show that $\int\varphi d\mu_\alpha \to \int\varphi d\mu$ for any $\varphi\in\Ccal(\Pbf^N_K)$, it suffices by a standard approximation argument  verify it for $\varphi$ in a dense subspace of $\Ccal(\Pbf^N_K)$.  By linearity and what we have already shown, one only needs to check it when $\varphi=\lambda_f$ for an arbitrary normalized homogeneous polynomial $f\in K[X]$, which holds by hypothesis.
\end{proof}

\noindent
{\em Remark.}  The main ingredient in the proof of Theorem~\ref{MainTheorem}, namely the density of the subalgebra $\Acal(\Pbf^N_K)$ in $\Ccal(\Pbf^N_K)$, is similar in spirit to a density result of Gubler (\cite{GublerCrelle} Thm.~7.12).  Specifically, working over an arbitrary compact Berkovich analytic space $X$, Gubler considers the space of functions $z\mapsto-\log\|1(z)\|^{1/m}$, for integers $m\geq1$ and a certain class of algebraic metrics $\|\cdot\|$ defined on the trivial line bundle $O_{X}$.  After showing that the space of all such functions is point-separating and closed under taking maximums and minimums, he appeals to the lattice form of the Stone-Weierstrass theorem to show that this space is dense in the space $\Ccal(X)$ of continuous functions.  

Working only over $\Pbf^N_K$, our density result is rather simpler than Gubler's.  Given a normalized homogeneous polynomial $f\in K[X]$ of degree $d$, we may view it as a section $f\in\Gamma(\PP^N,\Ocal(d))$, and we may therefore write $\lambda_f(z)=\|f(z)\|_\sup$ where $\|\cdot\|_{\sup}$ is the sup-metric on $\Ocal(d)$.   Taking advantage of the identity $\lambda_f\lambda_g=\lambda_{fg}$, we use the multiplicative form of the Stone-Weierstrass theorem to obtain the density of the algebra $\Acal(\Pbf^N_K)$ generated by the functions $\lambda_f$.

\subsection{}  Theorem~\ref{MainTheorem} was stated in terms of the weak convergence of nets of arbitrary unit Borel measures on $\Pbf^N_K$, but our principal concern is with the more specific notion of equidistribution.  Given a finite multiset $Z$ of points in $\Pbf^N_K$, define a unit Borel measure $\delta_{Z}$ on $\Pbf^N_K$ by
\begin{equation*}
\delta_{Z} = \frac{1}{|Z|}\sum_{z\in Z}\delta_z.
\end{equation*}
Here $\delta_z$ is the unit Dirac measure at $z$, characterized by the formula $\int\varphi d\delta_z=\varphi(z)$ for all $\varphi\in \Ccal(\Pbf^N_K)$.  Since $Z$ is a multiset, we understand the cardinality $|Z|$ and the sum over $z\in Z$ to be computed according to multiplicity.  Given a net $\langle Z_\alpha\rangle$ of finite multisets in $\Pbf^N_K$, and a unit Borel measure on $\Pbf^N_K$, we say that the net $\langle Z_\alpha\rangle$ is $\mu$-equidistributed if $\delta_{Z_\alpha}\to\mu$ weakly.

\begin{cor}\label{MainCor}
Let $\langle Z_\alpha\rangle$ be a net of finite multisets in $\Pbf^N_K$, and let $\mu$ be a unit Borel measure on $\Pbf^N_K$.  Then $\langle Z_\alpha\rangle$ is $\mu$-equidistributed if and only if $\frac{1}{|Z_\alpha|}\sum_{z\in Z_\alpha}\lambda_f(z)\to\int\lambda_f d\mu$ for all normalized homogeneous polynomials $f\in K[X]$.
\end{cor}

\section{Equidistribution and Reduction}\label{ReductionSect}

\subsection{}  Let $r:\PP^{N}(K)\to\PP^N(\tilde{K})$ be the usual reduction map on the ordinary projective space; thus $r(a_0:a_1:\dots:a_N)=(\tilde{a}_0:\tilde{a}_1:\dots:\tilde{a}_N)$, where homogeneous coordinates have been chosen so that $\max\{|a_0|,|a_1|,\dots,|a_N|\}=1$.  In $\S$~\ref{ReductionDefSect} we will discuss how this map extends naturally to a reduction map $r:\Pbf^{N}_K\to\PP^{N}_{\tilde{K}}$ from the Berkovich projective space $\Pbf^{N}_K$ onto the scheme-theoretic projective space $\PP^{N}_{\tilde{K}}$ over the residue field $\tilde{K}$.

Recall that $\zeta_{0,1}$ denotes the point of $\Abf_K^{N+1}$ corresponding to the supremum norm on the polydisc $D(0,1)$ in $\KK^{N+1}$ with center $0=(0,\dots,0)$ and polyradius $1=(1,\dots,1)$, as discussed in $\S$~\ref{AffineSubSect}.  Let $\gamma$ be the point of $\Pbf^N_K$ defined by $\gamma=\pi(\zeta_{0,1})$, and let $\delta_\gamma$ be the unit Dirac measure supported at $\gamma$.  In this section we will give a useful necessary and sufficient condition for $\delta_\gamma$-equidistribution in terms of the functions $\lambda_f$.  In $\S$~\ref{DeltaGammaEquiSect} we will use this result to establish the $\delta_\gamma$-equidistribution of a net $\langle Z_\alpha\rangle$ in $\Pbf^N_K$, provided the image of $\langle Z_\alpha\rangle$ under the reduction map $r:\Pbf^{N}_K\to\PP^{N}_{\tilde{K}}$ satisfies a certain weak Zariski-density property.

We begin with a well-known lemma which records a basic property of the seminorm $[\cdot]_{\zeta_{0,1}}$.  Given a polynomial $f\in K[X]$, recall that $H(f)$ denotes the maximum absolute value of the coefficients of $f$.

\begin{lem}\label{GLHomogeneous}
The identity $[f]_{\zeta_{0,1}}=H(f)$ holds for all $f\in K[X]$.
\end{lem}
\begin{proof}
This is trivial if $f=0$, so we may assume $f\neq0$.  Scaling $f$ by an appropriate element of $K$, we may assume without loss of generality that $f$ is normalized, and thus we must show that $[f]_{\zeta_{0,1}}=1$.  Plainly $|f(a)|\leq1$ for all $a\in D(0,1)$ by the ultrametric inequality, whereby $[f]_{\zeta_{0,1}}\leq 1$.  Since $f$ has coefficients in $K^\circ$, and thus in $\KK^\circ$, it reduces to a polynomial $\tilde{f}(X)\in\tilde{\KK}[X]$.  Since $H(f)=1$, the reduced polynomial $\tilde{f}(X)$ is nonzero and therefore is nonvanishing on a nonempty Zariski-open subset of $\tilde{\KK}^{N+1}$ (note that $\tilde{\KK}$ is algebraically closed).  Select some $\tilde{a}_0\in \tilde{\KK}^{N+1}$ such that $\tilde{f}(\tilde{a}_0)\neq0$, and let $a_0\in D(0,1)$ be a point which reduces to $\tilde{a}_0$.  Then $1=|f(a_0)|\leq [f]_{\zeta_{0,1}}$, completing the proof that $[f]_{\zeta_{0,1}}=1$.
\end{proof}

\noindent
{\em Remark.}  It follows from Lemma~\ref{GLHomogeneous} and the multiplicativity of the seminorm $[\cdot]_{\zeta_{0,1}}$ that $H(fg)=H(f)H(g)$ for any two polynomials $f,g\in K[X]$; this fact is essentially equivalent to Gauss's lemma from algebraic number theory.  Consequently, $\zeta_{0,1}$ is commonly called the Gauss point of $\Abf^{N+1}_K$, and likewise $\gamma=\pi(\zeta_{0,1})$ the Gauss point of $\Pbf^{N}_K$.

\begin{thm}\label{DeltaGammaEquiTheorem}
Let $\langle Z_\alpha\rangle$ be a net of nonempty finite multisets in $\Pbf^N_K$.  Then $\langle Z_\alpha\rangle$ is $\delta_{\gamma}$-equidistributed if and only if the limit
\begin{equation}\label{DeltaGammaLimit}
\lim_\alpha \frac{|\{z\in Z_\alpha\mid\lambda_f(z)<t\}|}{|Z_\alpha|}=0
\end{equation}
holds for each normalized homogeneous polynomial $f\in K[X]$ and each real number $0<t<1$.
\end{thm}
\begin{proof}
Given a normalized homogeneous polynomial $f\in K[X]$, we have $0\leq \lambda_f(z)\leq1$ for all $z\in\Pbf^{N}_K$ (the upper bound following from the ultrametric inequality), and $\lambda_f(\gamma)=[f]_{\zeta_{0,1}}/\|\zeta_{0,1}\|^{\deg(f)}=1$ (by Lemma~\ref{GLHomogeneous}).  For each multiset $Z_\alpha$, define the sum
\begin{equation*}
S_\alpha(f)=\frac{1}{|Z_\alpha|}\sum_{z\in Z_\alpha}\lambda_f(z).
\end{equation*}
Plainly $0\leq S_\alpha(f)\leq1$, and it follows from the above observations and Corollary~\ref{MainCor} that $\langle Z_\alpha\rangle$ is $\delta_{\gamma}$-equidistributed if and only if $S_\alpha(f)\to1$ for all normalized homogeneous polynomials $f\in K[X]$.  In order to prove the theorem, it therefore suffices to show that $S_\alpha(f)\to1$ if and only if the limit $(\ref{DeltaGammaLimit})$ holds for all $0<t<1$.

Assuming that the limit $(\ref{DeltaGammaLimit})$ holds for all $0<t<1$, we have
\begin{equation*}
S_\alpha(f)\geq t\frac{|\{z\in Z_\alpha\mid\lambda_f(z)\geq t\}|}{|Z_\alpha|} = t\bigg(1-\frac{|\{z\in Z_\alpha\mid\lambda_f(z)<t\}|}{|Z_\alpha|}\bigg)\to t.
\end{equation*}
As $0<t<1$ is arbitrary and $S_\alpha(f)\leq1$, we deduce that $S_\alpha(f)\to1$.

Conversely, suppose that 
\begin{equation}\label{DeltaGammaLimitFail}
\limsup_\alpha \frac{|\{z\in Z_\alpha\mid\lambda_f(z)<t\}|}{|Z_\alpha|}=\epsilon>0.
\end{equation}
for some $0<t<1$.  Passing to a subnet, we may assume that this limsup is actually a limit.  Using the fact that $0\leq \lambda_f(z)\leq1$ for all $z\in\Pbf^{N}_K$, we have
\begin{equation*}
\begin{split}
S_\alpha(f) & = \frac{1}{|Z_\alpha|}\sum_{\stackrel{z\in Z_\alpha}{\lambda_f(z)<t}}\lambda_f(z) + \frac{1}{|Z_\alpha|}\sum_{\stackrel{z\in Z_\alpha}{\lambda_f(z)\geq t}}\lambda_f(z) \\
	& \leq t\frac{|\{z\in Z_\alpha\mid\lambda_f(z)< t\}|}{|Z_\alpha|} + \frac{|\{z\in Z_\alpha\mid\lambda_f(z)\geq t\}|}{|Z_\alpha|} \\
	& \to t\epsilon + (1-\epsilon)<1,
\end{split}
\end{equation*}
which means that $S_\alpha(f)\not\to1$ in this case.
\end{proof}

\subsection{}\label{ReductionDefSect}  Let $\AA^{N+1}_{\tilde{K}}=\Spec(\tilde{K}[X])$ and $\PP^{N}_{\tilde{K}}=\Proj(\tilde{K}[X])$ be the usual scheme-theoretic affine and projective spaces over the residue field $\tilde{K}$, as defined say in \cite{Hartshorne} $\S$~II.2.  Let $\pi_\sch:\AA^{N+1}_{\tilde{K}}\setminus\{\fp_0\}\to\PP^{N}_{\tilde{K}}$ denote the standard projection map, where $\fp_0$ denotes the ideal of polynomials in $\tilde{K}[X]$ vanishing at $0=(0,0,\dots,0)$.

Recall from $\S$~\ref{BerkProjSect} that the set $ \Sbf_K^N=\{\zeta\in\Abf^{N+1}_K\mid \|\zeta\|= 1\}$ can be taken as a domain for the quotient map $\pi: \Sbf_K^N\to\Pbf^{N}_K$.  Given a point $\zeta\in  \Sbf_K^N$, the ultrametric inequality implies that $[f]_\zeta\leq 1$ for all $f\in K^\circ[X]$.  Define
\begin{equation*}
\wp_\zeta = \{f\in K^\circ[X] \mid [f]_\zeta<1\},
\end{equation*}
and define $\tilde{\wp}_\zeta$ to be the image of $\wp_\zeta$ under the reduction map $K^\circ[X]\to\tilde{K}[X]$.  Then $\wp_\zeta$ is a prime ideal of the ring $K^\circ[X]$ (since $[\cdot]_\zeta$ is multiplicative), and $\tilde{\wp}_\zeta$ is a prime ideal of the ring $\tilde{K}[X]$ (since $\tilde{K}[X]\simeq K^\circ[X]/K^{\circ\circ}$).  We obtain affine and projective reduction maps
\begin{equation}\label{RedMaps}
\begin{CD}
 \Sbf_K^N  @>  >>   \AA^{N+1}_{\tilde{K}}\setminus \{\fp_0\}  \\ 
@V \pi VV                                    @VV \pi_{\sch} V \\ 
\Pbf^{N}_K        @> r >>      \PP^{N}_{\tilde{K}}
\end{CD}
\end{equation}
Here the affine reduction map $\Sbf_K^N  \to \AA^{N+1}_{\tilde{K}}\setminus \{\fp_0\}$ is given by $\zeta\mapsto\tilde{\wp}_\zeta$, and to see that there exists a unique map $r:\Pbf^{N}_K\to\PP^{N}_{\tilde{K}}$ completing the commutative diagram $(\ref{RedMaps})$, it suffices to show that $\pi_\sch(\tilde{\wp}_\zeta)=\pi_\sch(\tilde{\wp}_\xi)$ whenever $\pi(\zeta)=\pi(\xi)$; this is straightforward to check using the definitions of $\pi$ and $\pi_\sch$.  It is also a standard exercise in the definitions to show that the reduction map $r$ is surjective (see \cite{Berkovich} $\S$~2.4).

The Gauss point $\gamma=\pi(\zeta_{0,1})$ can be characterized as the unique point of $\Pbf^{N}_K$ which reduces to the generic point of $\PP^{N}_{\tilde{K}}$.  To see this, note that $\tilde{\wp}_{\zeta_{0,1}}$ is the zero ideal of $\tilde{K}[X]$ by Lemma~\ref{GLHomogeneous}; thus $r(\gamma)=\pi_\sch(\tilde{\wp}_{\zeta_{0,1}})$ is the generic point of $\PP^{N}_{\tilde{K}}$.  Conversely, if $\zeta\in \Sbf_K^N$ and $\pi_\sch(\tilde{\wp}_{\zeta})$ is the generic point of $\PP^{N}_{\tilde{K}}$, then the ideal $\tilde{\wp}_{\zeta}$ contains no nonzero homogeneous polynomials in $\tilde{K}[X]$.  In other words, $\lambda_f(\pi(\zeta))=[f]_\zeta=1$ for all normalized homogeneous polynomials $f\in K[X]$.  Since $\lambda_f(\gamma)=1$ for all such $f$ as well, and since the functions $\lambda_f$ separate the points of $\Pbf^{N}_K$, we conclude that $\pi(\zeta)=\gamma$.

\subsection{}\label{DeltaGammaEquiSect}  Given a finite multiset $Z$ of points in $\Pbf^N_K$, define its reduction $\tilde{Z}$ to be the finite multiset in $\PP^N_{\tilde{K}}$ where the multiplicity of a point $\tilde{z}$ in $\tilde{Z}$ is the sum of the multiplicities of the points $z\in r^{-1}(\tilde{z})$ in $Z$.  Thus $|Z|=|\tilde{Z}|$.

Let $\langle \tilde{Z}_\alpha\rangle$ be a net of nonempty finite multisets in $\PP^N_{\tilde{K}}$.  We say that the net $\langle \tilde{Z}_\alpha\rangle$ is generic if, given any subnet $\langle \tilde{Z}_\beta\rangle$ of $\langle \tilde{Z}_\alpha\rangle$ and any proper Zariski-closed subset $W\subset\PP^N_{\tilde{K}}$, there exists $\beta_0$ such that $\tilde{Z}_\beta\cap W=\emptyset$ for all $\beta\geq\beta_0$.  We say that the net $\langle \tilde{Z}_\alpha\rangle$ is weakly generic if the limit
\begin{equation}\label{StronglyGeneric}
\lim_\alpha\frac{|\tilde{Z}_\alpha\cap W|}{|\tilde{Z}_\alpha|}=0
\end{equation}
holds for all proper Zariski-closed subsets $W\subset\PP^N_{\tilde{K}}$.  Note that a generic net is weakly generic: if $\langle \tilde{Z}_\alpha\rangle$ is generic, then $|\tilde{Z}_\alpha\cap W|=0$ for all sufficiently large $\alpha$, whereby $|\tilde{Z}_\alpha\cap W|/|\tilde{Z}_\alpha|\to0$.

\begin{thm}\label{MainEquiTheorem}
Let $\langle Z_\alpha\rangle$ be a net of nonempty finite multisets in $\Pbf^N_K$.  If the reduction $\langle \tilde{Z}_\alpha\rangle$ is weakly generic in $\PP^N_{\tilde{K}}$, then $\langle Z_\alpha\rangle$ is $\delta_{\gamma}$-equidistributed.
\end{thm}
\begin{proof}
Given a normalized homogeneous polynomials $f\in K[X]$, let $\tilde{f}\in\tilde{K}[X]$ denote its reduction, and define 
\begin{equation*}
V(\tilde{f}) = \{\fp\in\PP^N_{\tilde{K}} \mid (\tilde{f})\subseteq\fp\}
\end{equation*}
to be the hypersurface in $\PP^N_{\tilde{K}}$ associated to $\tilde{f}$.  Given a multiset $Z_\alpha$ and a point $z\in Z_\alpha$, by Lemma~\ref{RestSurjLem} we may select $\zeta_z\in\Abf^{N+1}_K$ such that $\pi(\zeta_z)=z$ and $\|\zeta_z\|=1$.  It follows from the assumptions $H(f)=1$ and $\|\zeta_z\|=1$, along with the ultrametric inequality, that $\lambda_f(z)=[f]_{\zeta_z}\leq1$.  Moreover, it is easy to check using the definition of the reduction map that $\lambda_f(z)=[f]_{\zeta_z}<1$ if and only if $\tilde{z}=r(z)$ is contained in the hypersurface $V(\tilde{f})$.  Thus
\begin{equation}\label{MainEquiLimit}
\frac{|\{z\in Z_\alpha\mid\lambda_f(z)<1\}|}{|Z_\alpha|}=\frac{| V(\tilde{f})\cap\tilde{Z}_\alpha|}{|\tilde{Z}_\alpha|}.
\end{equation}
Since $\langle \tilde{Z}_\alpha\rangle$ is weakly generic in $\PP^N_{\tilde{K}}$, the right-hand-side of $(\ref{MainEquiLimit})$ goes to zero in the limit.  It follows that the limit $(\ref{DeltaGammaLimit})$ holds for all normalized homogeneous polynomials $f\in K[X]$ and each real number $0<t<1$, and we conclude that $\langle Z_\alpha\rangle$ is $\delta_{\gamma}$-equidistributed using Theorem~\ref{DeltaGammaEquiTheorem}.
\end{proof}

\noindent
{\em Remark.}  The converse of Theorem~\ref{MainEquiTheorem} is false.  For example, let $\langle x_\alpha\rangle$ be a net in $K$ such that $|x_\alpha|<1$ for all $\alpha$, but such that $|x_\alpha|\to1$.  Theorem~\ref{DeltaGammaEquiTheorem} implies that the net $\langle Z_\alpha\rangle$ of singleton sets $Z_\alpha=\{(1:x_\alpha)\}$ is $\delta_\gamma$-equidistributed.  But each point $(1:x_\alpha)$ reduces to the same point $(1:0)$ in $\PP^1(\tilde{K})$, so $\langle \tilde{Z}_\alpha\rangle$ is not weakly generic in $\PP^1_{\tilde{K}}$.  

In practice, however, the converse of Theorem~\ref{MainEquiTheorem} may hold for certain classes of nets $\langle Z_\alpha\rangle$ of particular interest.  We will see an example of this in the proof of Theorem~\ref{NonArchErg}.

\section{A Ergodic Equidistribution Theorem}\label{ErgodicSect}

\subsection{}  Define the unit torus $\TT^N_K$ in $\PP^N(K)$ by
\begin{equation*}
\TT^N_K=\{(a_0:a_1:\dots:a_N)\in\PP^N(K) \mid |a_0|=|a_1|=\dots=|a_N|=1\}.
\end{equation*}
Note that $\TT^N_K$ is a group under coordinate multiplication, with neutral element $1=(1:\dots:1)$.  In this section we will prove an equidistribution result, in the case of residue characteristic zero, for the sequence $\langle a^\ell\rangle_{\ell=1}^{+\infty}$ formed by taking the powers of a point $a\in\TT^N_K$.

\subsection{}  We begin with some algebraic preliminaries.  Let $k$ be an arbitrary field of characteristic zero, and fix homogeneous coordinates $(x_0:x_1:\dots:x_N)$ on $\PP^N$ over $k$.  We identify the group variety $\GG_m^N$ over $k$ with the subvariety of $\PP^N$ defined by $x_0x_1\dots x_N\neq0$; the group law on $\GG_m^N(k)$ is given by coordinate multiplication, with neutral element $1=(1:\dots:1)$.  

Given an arbitrary point $a\in \GG_m^N(k)$, we denote by $a_1,\dots,a_N$ the unique elements of $k^\times$ such that $a=(1:a_1:\dots:a_N)$.  We say that $a$ is degenerate if the elements $a_1,\dots,a_N$ are multiplicatively dependent in $k^\times$; otherwise we say that $a$ is non-degenerate.  

Consider a subgroup $\Lambda$ of the integer lattice $\ZZ^N$.  The group $\Lambda$ gives rise to an algebraic subgroup $G_\Lambda$ of $\GG_m^N$ by
\begin{equation*}
G_\Lambda(k)=\{a\in \GG_m^N(k) \mid a_1^{\ell_1}\dots a_N^{\ell_N}=1\text{ for all }(\ell_1,\dots,\ell_N)\in\Lambda\}.
\end{equation*}
Conversely, all algebraic subgroups of $\GG_m^N$ arise in this way (\cite{BombieriGubler} $\S$3.2).  It follows from this correspondence that a point $a\in\GG_m^N(k)$ is degenerate if and only if $a$ is contained in some proper algebraic subgroup of $\GG_m^N$.  Moreover, it is clear that $a$ is degenerate if and only if $a^\ell$ is degenerate for all nonzero $\ell\in\ZZ$.

\begin{prop}\label{NonDegenerateProp}
A point $a\in \GG_m^N(k)$ is non-degenerate if and only if, for each proper Zariski-closed subset $W$ of $\GG_m^N(k)$, there exist only finitely many integers $\ell$ such that $a^\ell\in W$.
\end{prop}
\begin{proof}
The ``if'' direction is trivial.  For if $a$ is degenerate then $a\in G(k)$ for some proper algebraic subgroup $G$ of $\GG_m^N$, and therefore $a^\ell \in G(k)$ for all $\ell\in\ZZ$.

To prove the ``only if'' direction, assume that $a$ is non-degenerate, and consider a proper Zariski-closed subset $W$ of $\GG_m^N(k)$.  Denote by $a^\ZZ$ the cyclic subgroup of  $\GG_m^N(k)$ generated by $a$.  Since $a$ is non-degenerate, it is non-torsion, and therefore in order to complete the proof it is enough to show that $a^\ZZ\cap W$ is finite.  Replacing $W$ with the Zariski-closure of $a^\ZZ\cap W$, we may assume without loss of generality that $a^\ZZ\cap W$ is Zariski-dense in $W$.  A result of Laurent (\cite{Laurent}; see also \cite{BombieriGubler} Thm.~7.4.7) implies that, since $a^\ZZ\cap W$ is Zariski-dense in $W$, we must have $W = \cup_{j=1}^Jy_jG_j$ for some finite set $y_1,\dots,y_J\in\GG_m^N(k)$ and some finite collection $G_1,\dots,G_J$ of proper algebraic subgroups of $\GG_m^N$.  Therefore, in order to show that $a^\ZZ\cap W$ is finite, it suffices to show that $a^\ZZ\cap yG(k)$ is finite for an arbitrary $y\in\GG_m^N(k)$ and an arbitrary proper algebraic subgroup $G$ of $\GG_m^N$.  In fact, $a^\ZZ\cap yG(k)$ can contain at most one point.  For if $a^{\ell}$ and $a^{\ell'}$ are elements of $yG(k)$ for some distinct integers $\ell, \ell'\in\ZZ$, then $a^{\ell-\ell'}\in G(k)$, implying that $a^{\ell-\ell'}$ is degenerate.  This contradicts the assumption that $a$ is non-degenerate.
\end{proof}

\noindent
{\em Remark.}  The result of Laurent used in this proof is the $\GG_m^N$ case of what is commonly called the Lang (sometimes Mordell-Lang) conjecture.  It is now a much more general theorem, holding for finite-rank subgroups of semi-abelian varieties, due to Laurent, Faltings, Vojta, and McQuillen; see \cite{HindrySilvermanBook} $\S$~F.1.1 for a survey.

\subsection{}  We return to our complete non-archimedean field $K$.  Observe that, given a point $a\in\PP^N(K)$, we have $a\in\TT^N_K$ if and only if $\tilde{a}\in\GG_m^N(\tilde{K})$.  

\begin{thm}\label{NonArchErg}
Assume that the residue field $\tilde{K}$ has characteristic zero.  Let $a\in\TT^N_K$, and for each integer $\ell\geq1$, define $Z_\ell=\{a,a^2,\dots,a^\ell\}$, considered as a multiset in $\TT^N_K\subset\PP^N(K)\subset\Pbf_K^N$ of cardinality $|Z_\ell|=\ell$.  The sequence $\langle Z_\ell\rangle_{\ell=1}^{\infty}$ is $\delta_\gamma$-equidistributed in $\Pbf_K^N$ if and only if the point $\tilde{a}$ is non-degenerate in $\GG_m^N(\tilde{K})$.
\end{thm}
\begin{proof}
Assume that $\tilde{a}$ is non-degenerate in $\GG_m^N(\tilde{K})$.  By Proposition~\ref{NonDegenerateProp}, $\tilde{a}^\ZZ\cap W$ is finite for any proper Zariski-closed subset $W$ of $\PP^N(\tilde{K})$, which implies that the sequence $\langle \tilde{Z}_\ell\rangle_{\ell=1}^{+\infty}$ is weakly generic in $\PP^N(\tilde{K})$.  By Theorem~\ref{MainEquiTheorem}, we conclude that $\langle Z_\ell\rangle_{\ell=1}^{\infty}$ is $\delta_\gamma$-equidistributed.

Conversely, suppose that $\tilde{a}$ is degenerate in $\GG_m^N(\tilde{K})$.  Writing 
\begin{equation*}
a=(1:a_1:\dots:a_N)\in\TT^N_K\subset\PP^N(K)
\end{equation*}
the fact that $\tilde{a}$ is degenerate means that $\tilde{a}_1^{\ell_{1}}\dots \tilde{a}_N^{\ell_{N}}=1$ in $\tilde{K}$ for a nonzero vector $(\ell_1,\dots,\ell_N)\in\ZZ^N$.  Therefore the element $A:=a_1^{\ell_{1}}\dots a_N^{\ell_{N}}-1\in K$ satisfies $|A|<1$.  Select an integer $r\geq0$ such that $\ell_n+r\geq0$ for all $1\leq n\leq N$, let $\ell=\sum_{n=1}^{N}\ell_n$, and define $f\in K[X]$ by 
\begin{equation*}
f(X)=X_1^{\ell_{1}+r}\dots X_N^{\ell_{N}+r} - X_0^\ell(X_1\dots X_N)^r.
\end{equation*}
Note that $f$ is nonzero, homogeneous, and satisfies $H(f)=1$.  In particular,
\begin{equation}\label{EquiFail1}
\int \lambda_fd\delta_\gamma=\lambda_f(\gamma)=[f]_{\delta_{0,1}}=H(f)=1.
\end{equation}
On the other hand, it is easy to check that $|f(1,a_1,\dots,a_N)|=|A|$, and that more generally $|f(1,a_1^j,\dots,a_N^j)|\leq|A|$ for all integers $j\geq1$.  Therefore
\begin{equation}\label{EquiFail2}
\frac{1}{|Z_\ell|}\sum_{z\in Z_\ell}\lambda_f(z)=\frac{1}{\ell}\sum_{j=1}^\ell|f(1,a_1^j,\dots,a_N^j)|\leq|A|<1.
\end{equation}
Letting $\ell\to+\infty$, $(\ref{EquiFail1})$ and $(\ref{EquiFail2})$ together show that the sequence $\langle Z_\ell\rangle_{\ell=1}^{\infty}$ fails the criterion for $\delta_\gamma$-equidistribution stated in Corollary~\ref{MainCor}.
\end{proof}

\subsection{}  Theorem~\ref{NonArchErg} is a non-archimedean analogue of the following classical equidistribution result of Weyl \cite{Weyl}.  As in the non-archimedean case, define the unit torus $\TT^N_\CC$ in $\PP^N(\CC)$ by
\begin{equation*}
\TT^N_\CC=\{(a_0:a_1:\dots:a_N)\in\PP^N(\CC) \mid |a_0|=|a_1|=\dots=|a_N|=1\}.
\end{equation*}
Then $\TT^N_\CC$ is a compact topological group, and as such it carries a unique normalized Haar measure.

\begin{thm}[Weyl]\label{WeylThm}
Let $a\in\TT^N_\CC$, and for each integer $\ell\geq1$, define $Z_\ell=\{a,a^2,\dots,a^\ell\}$, considered as a multiset in $\TT^N_\CC$ of cardinality $|Z_\ell|=\ell$.  The sequence $\langle Z_\ell\rangle_{\ell=1}^{\infty}$ is Haar-equidistributed in $\TT^N_\CC$ if and only if $a$ is non-degenerate in $\GG_m^N(\CC)$.
\end{thm}

An important difference between Theorems~\ref{NonArchErg} and \ref{WeylThm} stems from the fact that, in the non-archimedean case, the assumption that $\ch(\tilde{K})=0$ ensures that $\tilde{K}$ has infinitely many elements, which implies that the field $K$ is not locally compact.  Consequently, the unit torus $\TT^N_K$ is noncompact and thus has no Haar measure in the traditional sense.  On the other hand, observe that $\TT^N_K$ is contained in the compact Berkovich unit torus 
\begin{equation*}
\Tbf^N_K=\{\pi(\zeta)\mid\zeta\in \Abf^{N+1}_K\text{ and } [X_0]_\zeta=[X_1]_\zeta=\dots=[X_N]_\zeta=1\},
\end{equation*}
and that the Dirac measure $\delta_\gamma$ supported on the Gauss point $\gamma\in\Tbf^N_K$ is invariant under the translation action of the group $\TT^N_K$ on $\Tbf^N_K$.  Thus $\delta_\gamma$ is a natural substitute for Haar measure in this setting.

Analogues of Weyl's Haar-equidistribution result have been investigated over the locally compact non-archimedean field $\QQ_p$, at least in the case $N=1$; see Bryk-Silva \cite{BrykSilva} and Coelho-Parry \cite{CoelhoParry}.

As pointed out by the referee, several things can be said in the direction of Theorem~\ref{NonArchErg} when the residue field $\tilde{K}$ has characteristic $p\neq0$.  First, the ``only if'' direction of the theorem continues to hold, with the same proof.  If $\tilde{K}$ is algebraic over its prime field $\FF_p$, then the statement of Theorem~\ref{NonArchErg} holds trivially, since all points of $\GG_m^N(\overline{\FF}_p)$ are torsion and therefore degenerate.  Finally, the statement of Theorem~\ref{NonArchErg} continues to hold for an arbitrary field $K$ of residue characteristic $p\neq0$ in the one-dimensional case.  For observe that an element $\tilde{a}$ in $\GG_m^1(\tilde{K})$ is degenerate if and only if it is torsion.  Using the trivial fact that the cyclic subgroup $\tilde{a}^\ZZ$ of $\GG_m^1(\tilde{K})$ is either finite or Zariski-dense, there is no need for Laurent's theorem, and therefore no need to assume that $\ch(\tilde{K})=0$.  We do not know whether the statement of Theorem~\ref{NonArchErg} holds for residue characteristic $p\neq0$ and $N\geq2$.

Finally, we point out that Theorems~\ref{NonArchErg} and \ref{WeylThm} can be viewed as examples of a more general class of equidistribution results arising naturally in ergodic theory.  Let $\Xcal$ be a compact Hausdorff space, let $T:\Xcal\to\Xcal$ be an automorphism, let $\mu$ be a $T$-invariant unit Borel measure on $\Xcal$, and let $x\in\Xcal$ be a point.  One of the basic goals of ergodic theory is to establish conditions under which the sequence $\langle Z_\ell\rangle_{\ell=1}^{+\infty}$ of multisets $Z_\ell=\{T(x),T^2(x),\dots,T^\ell(x)\}$ is $\mu$-equidistributed; see for example Furstenberg \cite{Furstenberg} and Lindenstrauss \cite{Lindenstrauss} for discussions of such results with a particular eye toward arithmetic applications.  

Weyl's original proof of Theorem~\ref{WeylThm} uses Fourier analysis, but there exists an alternate, ergodic-theoretic proof, see Furstenberg \cite{Furstenberg} Ch. 3.  It would be interesting to pursue non-archimedean equidistribution results such as Theorem~\ref{NonArchErg} from an ergodic-theoretic angle.  In view of Proposition~\ref{NonDegenerateProp} and it's reliance on Laurent's theorem \cite{Laurent}, it is especially intriguing to consider the possibility of deeper connections between ergodic theory and questions of Mordell-Lang type.

\medskip

\end{document}